\documentclass[a4paper,11pt]{article}

\usepackage[pdftex]{graphicx}
\usepackage{verbatim}
\usepackage{amsthm}
\usepackage{epsfig}
\usepackage{cite}
\usepackage{graphicx}
\usepackage{amsmath}
\usepackage{amssymb}
\usepackage{latexsym}
\usepackage[all]{xy}
\usepackage{amsfonts}

\theoremstyle{plain}
\newtheorem{thm}{Theorem}[section]
\newtheorem{cor}[thm]{Corollary}
\newtheorem{lem}[thm]{Lemma}
\newtheorem{prop}[thm]{Proposition}
\theoremstyle{definition}
\newtheorem{rem}[thm]{Remark}
\newtheorem{defn}[thm]{Definition}
\newtheorem{exm}[thm]{Example}

\newcommand{\FF}{\mathcal F}
\newcommand{\EE}{\mathcal E}
\newcommand{\N}{\mathbb N}
\newcommand{\GF}{\mathbb F}
\newcommand{\OO}{\mathcal O}
\newcommand{\PP}{\mathcal P}
\newcommand{\Places}{\mathbb P}

\title{ On the Invariants of Towers of Function Fields over Finite Fields }
\author{Florian Hess, Henning Stichtenoth and  Seher Tutdere}

\begin{document}
\sloppy
\maketitle

\begin{abstract} We consider a tower of function fields $\mathcal{F}=(F_n)_{n\geq 0}$ over a finite field $\mathbb{F}_q$ and a finite extension $E/F_0$ such that the sequence $\mathcal{E}:=E\cdot\mathcal{F}=(EF_n)_{n\geq 0}$ is a tower over the field $\mathbb{F}_q$. Then we deal with the following: What can we say about the invariants of $\mathcal{E}$; i.e., the asymptotic number of the places of degree $r$ for any $r\geq 1$ in $\mathcal{E}$,
if those of $\mathcal{F}$ are known?  We give a method based on explicit extensions for constructing towers of function fields over $\GF_q$ with finitely many prescribed invariants being positive, and towers of function fields over $\GF_q$, for $q$ a square, with at least one positive invariant and  certain prescribed invariants being zero. We show the existence of recursive towers attaining the Drinfeld-Vladut bound of order $r$,  for any $r\geq 1$ with $q^r$ a square, see \cite[Problem-2]{BR011}. Moreover, we give some examples of recursive towers with all but one invariants equal to zero.
\end{abstract}

\section{Introduction}
\paragraph{}

M. A. Tsfasman\cite{T092} introduced the notion of asymptotically exact sequences of function fields over finite fields  and studied on the invariants
 \[\beta_r(\FF):=\lim_{n\rightarrow \infty}(\textrm{the number of places of $F_n/\GF_q$ of degree $r$})/(\textrm{genus of $F_n$})\] 
 for any such sequence $\FF=(F_n)_{n\geq 0}$ with $r\geq 1$. 
The sequences for which $\beta_r$ exists and big are useful in information theory to obtain good algebraic geometric codes and  bounds for multiplication complexity in finite fields. S. Ballet and R. Rolland \cite{BR011} showed that these particular sequences have large asymptotic class number. In particular, one is interested in the exact sequences with small \textit{deficiency}, i.e., the difference between the right hand side and the left hand side of the inequality (\ref{bound}), which is related to the limit distribution of zeroes of zeta functions, see \cite{TV002}.  Lebacque\cite{L1009} gave an asymptotic estimate for the deficiency of some towers of function fields. 
 In 2007, Hasegawa\cite{H2007} and Lebacque\cite{L1007} independently  gave a proof of the existence of towers of  function fields (which are asymptotically exact sequences, see Theorem \ref{thmprem}(i)) with finitely many prescribed invariants $\beta_r>0$, using class field theory. However, in  \cite[p.64]{LZ011}), it is mentioned  that to  find towers of function fields with at least one nonzero invariant and certain prescribed invariants being zero is more difficult.
  
 The following open problem is stated in \cite{BR011}: Find asymptotically exact sequences of function fields over any finite field $\GF_q$,  attaining the Drinfeld-Vladut bound of order $r$ for any $r>2$ (except in the case $r=4$ and $q=2$, which is solved in \cite{BR011}). Note that when $q$ is a square and $r=1$, there are several examples, namely optimal towers, see \cite{GS095}. S. Ballet and R. Rolland \cite{BR011}  proved that for any prime power $q$ there exists a tower attaining the Drinfeld-Vladut bound of order $2$. Moreover, it is clear that the sequences attaining the Drinfeld-Vladut bound for some $r$ have deficiency zero. However, apart from the case that $q$ is a square, it is not known whether there are any exact sequences with \textit{deficiency} zero. We also compute here the exact value of \textit{deficiency} in all examples.

 The organization of the paper is as follows.
 
  In Section 2, we recall the basic definitions and introduce the notations. 
 
 In Section 3, we give some bounds for the invariants of towers of function fields over finite fields. Then we prove the existence of towers with  finitely many prescribed  invariants being positive and give a method for the construction of such towers, by using explicit extensions. Moreover, we prove that one can construct towers over $\GF_q$, for $q$ a square, with at least one positive invariant and certain prescribed invariants being zero.
 
  In Section 4,  we give some examples of non-optimal recursive towers with all but one invariants equal to zero. This is analog to the the open problem  given in \cite[p.3]{L1010}: Are there any infinite number fields (i.e., towers of number fields) with all but one invariants equal to zero. 
Moreover, we show that for any $r\geq 1$, prime power $q$ with $q^r$ a square, there are recursive towers of function fields over $\mathbb{F}_q$ attaining the Drinfeld-Vladut bound of order $r$. 

In the last section,  we give some new observations mainly concerning the quantity $A_r(q)$, namely the $r$-th Ihara's constant (see Definition \ref{Ihara_constant}),  for any prime power $q$ and positive integer $r$.

\section{Preliminaries}
\paragraph{}
Throughout this paper, we use basic facts and notations as in \cite{Stichtenoth}. For an algebraic function field $F/\mathbb{F}_q$  (with the finite field $\mathbb{F}_q$ as its full constant field), we denote by $N(F)$, $g(F)$ and $\mathbb{P}(F)$ the number of degree one places, the genus and the set of all places of $F/\mathbb{F}_q$, respectively.
For $r\geq 1$, define 
\[B_r(F):=\#\left\{ P\in \mathbb{P}(F)| \deg P=r   \right\}.\]
In particular, $B_1(F)=N(F)$. In \cite{T092}, M. A. Tsfasman studied asymptotic properties of the numbers $B_r(F)$ in sequences of function fields over $\mathbb{F}_q$. Specifically, he introduced the following concept:
 \begin{defn}\label{defnprem}
             A sequence $S=(F_n)_{n\geq 0}$ of function fields $F_n/\mathbb{F}_q$ is called asymptotically exact, if $g(F_n)\rightarrow \infty$ as $n\rightarrow 
             \infty$, and for all $r\geq 1$ the limit
             \[\beta_r(S):=\lim_{n\rightarrow \infty}\frac{B_r(F_n)}{g(F_n)} \]
             exists. 
\end{defn}
 For those numbers, one obtain the following bound \cite[Corollary 1]{T092}, \cite[Theorem 3]{S083}:
\begin{thm}[\textbf{Generalized Drinfeld-Vladut bound}]\label{DV}
           For an asymptotically exact sequence $S$ of function fields over a finite field $\mathbb{F}_q$ the following holds:
           \begin{equation}\label {bound}
           \sum _{r=1}^{\infty} \frac{r\beta_r(\mathcal{F})}{q^{r/2}-1} \leq1.
           \end{equation}
\end{thm}
\begin{defn}\label{Ihara_constant} For every $r\geq 1$, the real number 
             \[ A_r(q):=\limsup_{g\rightarrow \infty} \frac{B_r(F)}{g}\]
             where $F$ runs over all function fields over $\mathbb{F}_q$ of genus $g>0$ is called the $r$-th Ihara's constant.
             
             Moreover, the difference between the right hand side and the left hand side of the inequality (\ref{bound}) is called the \textit{deficiency} of the sequence $\mathcal{F}$. This is related to the limit distribution of zeroes of zeta functions, for details see \cite{TV002}. 
\end{defn}
As a consequence of Theorem \ref{DV}, one has 
\begin{cor} 
           \[A_r(q)\leq \frac{q^{r/2}-1}{r}.\]
\end{cor}

In this paper we will consider specific sequences of function fields over $\mathbb{F}_q$, namely towers. We will show that they are asymptotically exact, and we will study their invariants defined in Definition \ref{maindefn}.

An infinite sequence $\mathcal{F}=(F_n)_{n\geq 0}$ of function fields $F_n/\mathbb{F}_q$  is called a \textit{tower} over $\mathbb{F}_q$, if
\[F_0\subsetneqq F_1\subsetneqq F_2 \subsetneqq\ldots,\]
all extensions $F_{i+1}/F_i$ are finite separable, and $g(F_n)\rightarrow \infty$ as $n\rightarrow \infty$. 

\begin{prop}\label{propprem}
              Let $\mathcal{F}=(F_n)_{n\geq 0}$ be a tower over $\mathbb{F}_q$, $P\in \mathbb{P}(F_0)$ and $r\geq 1$. Set 
              \[ B_r(P,F_n):=\#\left\{ Q\in \mathbb{P}(F_n): Q|P \textrm{ and } \deg Q=r\right\}.\]
            Then the sequence 
             \[ \left( \frac{B_r(P,F_n)}{[F_n:F_0]}  \right)_{n\geq 0}\]
            is convergent. 
\end{prop}
\begin{proof} 
           Our proof is similar to T. Hasegawa's proof that the sequence \\$\left(B_r(F_n)/g(F_n)\right)_{n\geq 0}$ is convergent (cf.\cite[Proposition
            2.2]{H1007}). We proceed by induction over $r$. For $r=1$, the sequence $\left(B_1(P,F_n)/[F_n:F_0]\right)_{n\geq 0}$ is monotonically decreasing,
             and so convergent (cf.\cite[Lemma 7.2.3(a)]{Stichtenoth}). Now let $r\geq 1$ and assume that for all $1\leq s<r$, the sequence
              $\left( B_s(P,F_n)/[F_n:F_0]\right)_{n\geq 0}$ is convergent. Let $d:=\deg P$ with $d\nmid r$, then $B_r(P,F_n)=0$ for all $n\geq 0$. Hence, we can
            assume that $d\mid r$.
            
              Consider the constant field extension of $\mathcal{F}$ with the field $\mathbb{F}_{q^r}$; i.e., 
              \[ \mathcal{F}\cdot \mathbb{F}_{q^r}:=(F_n\cdot \mathbb{F}_{q^r})_{n\geq 0}.\]
            This is clearly a tower over $\mathbb{F}_{q^r}$. The place $P\in \mathbb{P}(F_0)$ splits into $P_1,\ldots,P_d \in \mathbb{P}(F_0\cdot
             \mathbb{F}_{q^r})$ of degree one, and all places of $F_n$ of degree $s\mid r$ split into $s$ degree one places of
              $F_n\cdot         \mathbb{F}_{q^r}/\mathbb{F}_{q^r}$. Hence, the following formula holds (cf. \cite[p. 206]{Stichtenoth}):
            \[ \sum_{s\mid r} s\cdot B_s(P,F_n)=\sum_{j=1}^{d}B_1(P_j,F_n\cdot \mathbb{F}_{q^r}).\]
            By induction hypothesis, the sequences
         \[ \left( \frac{B_s(P,F_n)}{[F_n:F_0]}  \right)_{n\geq 0}\textrm{ and } \left( \frac{B_1(P_j,F_n\cdot \mathbb{F}_{q^r})}{[F_n:F_0]}  \right)_{n\geq  0}\]            are convergent for $s<r$. Hence also the sequence $\left(B_r(P,F_n)/[F_n:F_0]\right)_{n\geq 0}$ converges. 
 \end{proof}
\begin{cor} \label{corprem} Let $\mathcal{F}=(F_n)_{n\geq 0}$ be a tower over $\mathbb{F}_q$, $P$ a place of $F_0$ and $r\geq 1$. Then the sequences
                        \[ \left(\frac{B_r(P,F_n)}{g(F_n)}\right)_{n\geq 0}, \left(\frac{B_r(F_n)}{[F_n:F_0]}\right)_{n\geq 0} \textrm{ and              }\left(\frac{B_r(F_n)}{g(F_n)}  \right)_{n\geq 0}\] 
                        are convergent. 
\end{cor}
\begin{proof}
            We recall that the sequence $\left(g(F_n)/[F_n:F_0]\right)_{n\geq 0}$ is convergent in $\mathbb{R}^+\cup \left\{ \infty \right\}$, and its limit \begin{equation}\label{eqnprem}
           \gamma(\mathcal{F}):=\lim_{n\rightarrow \infty} \frac{g(F_n)}{[F_n:F_0]} 
\end{equation}
              is called the genus of $\mathcal{F}$, see \cite[p.247 ]{Stichtenoth}.  The converges of the  sequence $\left(B_r(P,F_n)/g(F_n)\right)_{n\geq 0}$                  follows then immediately from Proposition \ref{propprem} and 
           \[ \frac{B_r(P,F_n)}{g(F_n)}=\frac{B_r(P,F_n)}{[F_n:F_0]}\cdot \frac{[F_n:F_0]}{g(F_n)}.\]
          Since 
         \[B_r(F_n)=\sum_{P\in \mathbb{P}(F_0)} B_r(P,F_n),\]
          also the other sequences in Corollary \ref{corprem} are convergent. 
\end{proof}
As a consequence, the following definitions make sense:
\begin{defn} \label{maindefn} 
             Let $\mathcal{F}=(F_n)_{n\geq 0}$ be a tower over $\mathbb{F}_q$, let $P\in \mathbb{P}(F_0)$ and $r\geq 1$. Define the real numbers
             \[ \nu_r(P,\mathcal{F}):=\lim_{n\rightarrow \infty} \frac{B_r(P,\mathcal{F})}{[F_n:F_0]},\quad \beta_r(P,\mathcal{F}):=\lim_{n\rightarrow \infty}                 \frac{B_r(P,F_n)}{g(F_n)},\]
             \[ \nu_r(\mathcal{F}):=\lim_{n\rightarrow \infty} \frac{B_r(\mathcal{F})}{[F_n:F_0]},\quad \beta_r(\mathcal{F}):=\lim_{n\rightarrow  \infty}\frac{B_r(F_n)}{g(F_n)}. \]
             We call $\nu_r(P,\mathcal{F})$ and  $\beta_r(P,\mathcal{F})$ \textit{local invariants} at $P$,  $\nu_r(\mathcal{F})$ and  $\beta_r(\mathcal{F})$
              \textit{ global invariants} of $\mathcal{F}$. Note that the definition of $\beta_r(\mathcal{F})$ is consistent with Definition \ref{defnprem}. 
          The sets 
          \[ Supp(\mathcal{F}):=\left\{ P\in \mathbb{P}(F_0): \nu_r(P,\mathcal{F})>0 \textrm{ for some $r\in \mathbb{N}$}\right\} \textrm{ and }\]
          \[\mathcal{P}(\mathcal{F}):=\left\{r\in \mathbb{N}: \; \nu_r(\mathcal{F})>0\right\}\]
            are called  the \textit{support} and the set of the \textit{positive  parameters} of $\mathcal{F}$, respectively.
\end{defn}
          We summarize as follows:
\begin{thm} \label{thmprem} 
           Let $\mathcal{F}=(F_n)_{n\geq 0}$ be a tower over $\mathbb{F}_q$. Then one has:
\begin{itemize}
           \item[(i)] For all $r\geq 1$, the limit
           \[ \beta_r(\mathcal{F}):=\lim_{n\rightarrow \infty} \frac{B_r(F_n)}{g(F_n)}\]
           exists; i.e., the tower is asymptotically exact. 
          \item[(ii)] (Generalized Drinfeld-Vladut bound and Deficiency) 
          \[ \sum_{r=1}^{\infty} \frac{r\beta_r(\mathcal{F})}{q^{r/2}-1} \leq 1,\]
          and the difference between the right hand side and the left hand side of this inequality is called the \textit{deficiency} of $\FF$.
          \item[(iii)] (Drinfeld-Vladut bound of order $r$) For all $r\geq 1$,
          \[ \beta_r(\mathcal{F})\leq A_r(q) \leq \frac{q^{r/2}-1}{r},\]
          where $A_r(q)$ is the $r$-th Ihara's constant. 
          \item[(iv)] Let $P\in \mathbb{P}(F_0)$ and $r\geq 1$. Then 
\[\beta_r(P,\mathcal{F})=\frac{\nu_r(P,\mathcal{F})}{\gamma(\mathcal{F})} \textrm{ and }  \beta_r(\mathcal{F})=\frac{\nu_r(\mathcal{F})}{\gamma(\mathcal{F})},\]
          where $\gamma(\mathcal{F})$ is the genus of the tower (see Equation (\ref{eqnprem})). 
          \item[(v)] For all $r\geq 1$, \[\nu_r(\mathcal{F})=\sum_{P\in \mathbb{P}(F_0)}\nu_r(P,\mathcal{F}) \textrm{ and }  \beta_r(\mathcal{F})=\sum_{P\in                  \mathbb{P}(F_0)}\beta_r(P,\mathcal{F}).\]
\end{itemize}
\end{thm}

Henceforth, we consider a tower $\mathcal{F}=(F_n)_{n\geq 0}$ of function fields over $\mathbb{F}_{q}$ and a finite separable extension $E$ of $F_0$. For convenience, we assume that $E,F_0,F_1,\ldots$ are all contained in a fixed algebraically closed field $\Omega$. For simplicity, we set $F:=F_0$ and denote by $\mathcal{E}:=E\cdot \mathcal{F}$ the sequence $\mathcal{E}=(E_n)_{n\geq 0}$, with $E_n:=EF_n$, of function fields over $\mathbb{F}_q$.

If furthermore the sequence $\mathcal{E}$ is a tower over $\mathbb{F}_q$ such that $E/F$ and $F_n/F$ are linearly disjoint for all $n\geq 1$, we             call $\mathcal{E}$ the \textit{ composite} tower of  $\mathcal{F}$ with $E/F$. We will here mainly be interested in the invariants of composite towers.
From now on, for any place $P\in \mathbb{P}(F)$ with an extension $Q$ in $E$, we denote by
\begin{itemize}
\item[-] $e(Q|P), f(Q|P), d(Q|P)$ the ramification index, relative degree and the different of $Q|P$, respectively, and 
\item[-]  $k(P)$ the residue class field of $P$. 
\end{itemize}
\section {Main Results} 
\paragraph{}
As for any tower $\EE/\GF_q$, the invariant $\beta_r(\EE)=\nu_r(\EE)/\gamma(\EE)$ for any $r\geq 1$, it is enough to estimate $\nu_r(\EE)$ and $\gamma(\EE)$. 
In Section 3.1 we assume that $\mathcal{E}:=E\cdot \mathcal{F}$ is a composite tower of a tower $\mathcal{F}$ with  $E/F$ of degree $m:=[E:F]$.
\subsection{\normalsize{Bounds for the invariants of a composite tower}}
We begin with a lemma concerning the splitting of places in the compositum of function fields, \cite[Proposition 3.9.6(a)]{Stichtenoth}.
\begin{lem}\label{split}
 Let $E/F$ and $F'/F$ be finite separable extensions of function fields contained in an algebraic closure of $F$. Suppose that $P$ is a place of $F$ which splits completely in the extension $F'$. Then every place $Q$ of $E$ lying above $P$ splits completely in the compositum $EF'$.
\end{lem}
\begin{prop}\label{mainprop}  For the composite tower $\mathcal{E}$, for any $s\geq 1$, we obtain 
\[\nu_s(\mathcal{E})\geq \#\left\{ Q\in \mathbb{P}(E)|\textrm{ $\deg Q=s$ and $Q\cap F$ splits completely in $\mathcal{F}$}\right\}.\]
\end{prop}
\begin{proof} Let $Q\in \mathbb{P}(E)$ and $P:=Q\cap F$ such that $P$ splits completely in $\mathcal{F}$. Then by Lemma \ref{split}, $Q$ splits completely in                    $E_n$ for all $n\geq 1$. Hence, 
\begin{equation*}
              B_s(Q,E_n)=[E_n:E] \textrm{ where $s=\deg Q$},
\end{equation*}
             which yields  $\nu_s(Q,\mathcal{E})=1$, and so by Theorem \ref{thmprem}(v) the proposition follows. 
\end{proof}
\begin{rem} For any $d\geq 1$ and $P\in \mathbb{P}(F)$, the following holds:
\begin{equation}
            \sum_{r=1}^m \sum _{\substack{Q\in \mathbb{P}(E)\\ Q|P,\; s=rd}} \nu_s(Q,\mathcal{E})\geq \nu_d(P,\mathcal{F}).
\end{equation}
\end{rem}
\begin{proof} 
            The proof follows from the following argument. Let $P_n\in \Places(F_n)$ (for any $n\geq 1$) lying above $P$ of $\deg P_n=d$ for some $d\geq 1$. Then for  any extension $Q_n$ of $P_n$ in $E_n$, we have $f(Q_n|P_n)=r$ for some $1\leq r\leq m$, and so $\deg Q_n=rd$.
\end{proof}
\begin{prop} \label{propcomp}  Let $Q\in \mathbb{P}(E)$ and $P:=Q\cap F$. Then for all $s>0$,
            \[ \nu_s(Q,\mathcal{E})\leq  \sum_{\substack{ d\in \mathcal{P}(\mathcal{F})\\ d|s,\;d\geq \frac{s}{m}}} \frac{md}{s} \nu_d(P,\mathcal{F}) \textrm{ and }        
            \nu_s(\mathcal{E})\leq  \sum_{\substack{d\in \mathcal{P}(\mathcal{F})\\d|s,\; d\geq \frac{s}{m}}} \frac{md}{s} \nu_d(\mathcal{F}).\]        
\end{prop}
\begin{proof} Let $Q_n$ be an extension  of $Q$ in $E_n$  of degree $s$ and $P_n:=Q_n\cap F_n$, for any $n\geq 1$. Then clearly $P_n|P$ and $\deg P_n=d$ with $d$ dividing $s$ and $d\geq \frac{s}{m}$, since $f(Q_n|P_n)\leq m$. Conversely, any place $P_n$ of $F_n$ lying above $P$ with $\deg P_n=d$, and satisfying $d\geq \frac{s}{m}$ has at most $\frac{md}{s}$ extensions of  degree $s$ in $E_n$, by using Fundamental Equality \cite{Stichtenoth}. Hence,
\begin{eqnarray}\label{limcomp}
             B_s(Q,E_n)\leq \sum_{\substack{d\in \N\\d|s,\;d\geq \frac{s}{m}} }\frac{md}{s} B_d(P,F_n)
\end{eqnarray}
           Then dividing by $[E_n:E]$ of both sides of (\ref{limcomp}) yields the bound for $\nu_s(Q,\mathcal{E})$. Then the desired bound for $\nu_s(\mathcal{E})$ follows, by using Theorem \ref{thmprem}(v). 
        
\end{proof}
\begin{cor}\label{corcomp} For the tower $\mathcal{E}$, we obtain that 
\begin{itemize}
           \item[(i)] $Supp(\mathcal{E})\subseteq \left\{ Q\in \mathbb{P}(E): \; Q\cap F\in Supp(\mathcal{F})\right\}$ and 
            \item[(ii)] if $\mathcal{P}(\mathcal{F})$ is finite, then $\mathcal{P}(\mathcal{E})$ is also finite.
\end{itemize}
\end{cor}
          Notice that as  for a given integer $r>0$ there are finitely many places of degree dividing $r$, if $\mathcal{P}(\mathcal{F})$ is finite, then the set            $Supp(\mathcal{F})$ is also finite. Furthermore,  when $\gamma(\mathcal{F})<\infty$, by Theorem \ref{thmprem}(iv), for any $r \in                                    \mathcal{P}(\mathcal{F})$, we have $\beta_r(\mathcal{F})>0$, and moreover, by Theorem \ref{genus}, $\gamma(\mathcal{E})<\infty$, and hence                             $\beta_s(\mathcal{E})>0$ for all $s\in \mathcal{P}(\mathcal{E})$.

         We also note here that until now there are no known asymptotically exact sequences of global fields with infinite set of positive  parameters. 
\subsection {\normalsize{Construction of composite towers with prescribed  invariants}}
\paragraph{}
         Now we will give a method for constructing towers with certain prescribed invariants.  We say that a tower $\FF$ containing $F$ is \textit{pure}, if for all $P\in \Places(F)$ and $r\in \N$, the inequality $\nu_r(P,\mathcal{F})> 0$ implies $\deg P=r$ and $\nu_s(P,\mathcal{F})=0$ for all $s\neq r$. 
         In this part, we will prove our main result: 
         
\begin{thm}\label{mainthm} Let $\mathcal{F}/F$ be a tower over $\GF_q$ with a finite support and let $N\subset \N$ be a non empty finite set. Then there exists a finite separable extension $E/F$ such that $\EE:=E\cdot \FF$ is a composite tower with
\begin{itemize}
\item[(i)] for all $s\in \N$,
\begin{equation*}
\nu_s(\EE)=\sum_{\substack{f\in N\\ d\in \PP(\FF)}} \frac{f}{s}\sum_{\substack{P\in Supp(\FF)\\\;\; lcm(f\deg P, d)=s}} d\cdot \nu_d(P,\FF)
\end{equation*}
and 
\begin{equation} 
\label{stm-1} Supp(\EE)=\big\{ Q\in \Places(E): Q\cap F\in Supp(\FF)\big\},
\end{equation}
\[\textrm{\hspace{-0.6cm}}\PP(\EE)=\big\{s \in \N: s=lcm(f\deg P,d) \textrm{ with $f\in N$, $d\in \N$, $P\in Supp(\FF)$\big\}}.\]
\item[(ii)]  If furthermore $\FF/F$ is \textit{pure}, then  for all $s\in \N$,
\begin{equation*}
 \nu_s(\EE)=\sum_{\substack{f\in N,\; d\in \PP(\FF) \\ fd=s}} \nu_d(\FF) \qquad \textrm{ and }
\end{equation*}
\begin{equation*}
 \PP(\EE)=\big\{ s\in \N: s=fd \textrm{ with $f\in N$, $d\in \PP(\FF)$\big\}}.
\end{equation*}
\end{itemize}
\end{thm}

  To be able to prove Theorem \ref{mainthm}, we begin with some results which will be used to construct an appropriate extension $E/F$ such that $\mathcal{E}:=E\cdot \mathcal{F}$ is a composite tower over $\mathbb{F}_q$ with  certain properties.        
\begin{prop}\label{crt}
          Let $F/\mathbb{F}_q$ be a function field with a finite set $S\subseteq \mathbb{P}(F)$ of pairwise distinct places, and a place $R\in                              \mathbb{P}(F)\setminus S$.  For each $P\in S$ let $N_P\subset \mathbb{N}$ be a finite set such that $\sum_{f\in N_P} f$ are equal for all $P\in S$.  
          Then there is a finite separable extension $E$ of $F$ such that 
\begin{itemize}
        \item[(i)] $[E:F]=m$ where $m:=\sum_{f\in N_P} f$, and $R$ is totally ramified in $E$.
        \item[(ii)] For each $P\in S$, $f\in N$, there exists exactly one extension $Q$ of $P$ in $E/\mathbb{F}_q$ with $f(Q|P)=f$.
        \item[(iii)] There is $y\in E$ such that $E=F(y)$ and $\big\{1,y,\ldots,y^{m-1}\big\}$ is an integral basis for $E/F$ at all $P\in S$. 
\end{itemize}
\end{prop}
\begin{proof}  For each $P\in S$, we set 
               \[\varphi_P(T):= \prod\limits_{\substack{f\in N_P}}g_f(T)=\sum _{k=0}^{m}{a}_{kP}T^k\in \mathcal{O}_P[T],\]
              where $g_f\in \mathcal{O}_P[T]$ is a monic polynomial which is irreducible over $k(P)$ of $\deg g_f=f$.
              Then by the Weak Approximation Theorem \cite{Stichtenoth}, for each $k=0,\ldots,m$, there exist elements  $b_1,\ldots,b_m \in F$ such that 
\begin{itemize}
           \item $v_{P}(b_i-a_{iP})>0$  for all $i=1,\ldots,m-1$ and  $P\in S$, and
           \item $v_R(b_m)=0$,  $gcd(m,v_R(b_0))=1$ and either 
           \[v_R(b_i) \geq v_R(b_0)>0 \textrm{ for $i=1,\ldots,m-1$}\qquad \textrm{ or }\qquad \qquad\]
           \[v_R(b_0)<0, \;v_R(b_i)\geq 0\textrm{ for $i=1,\ldots,m-1$}.\qquad\qquad \qquad\]
\end{itemize}
             Note that w.l.o.g we can take $b_m:=1$. Now we set $\varphi(T):=\sum _{k=0}^{m} b_k T^k\in \bigcap_{P\in S}\mathcal{O}_{P}[T]$. Then 
             \[ \varphi(T)\equiv \varphi_P(T)\textrm{ over $k(P) $ for  $P\in S$,} \textrm{ and } \]
            by the generalized  Eisenstein's Irreducibility Criterion \cite{Stichtenoth} with the place $R$, the polynomial $\varphi(T)$ is  irreducible over $F$.
            Set $E:=F(y)$ where $y$ is a root of $\varphi(T)$. Hence, $[E:F]=m$ and by the same irreducibility criterion, $R$ is totally ramified in $E$, and so                assertion (i) follows. Then by applying Kummer's Theorem \cite{Stichtenoth}, (ii) follows. Note that $E/F$ is separable, since by Kummer's Theorem               each $P\in S$ is unramified in $E$.  Then (iii) is clear from the factorization of $\varphi(T)$ over $k(P)$.

\end{proof}
\begin{rem} In Proposition \ref{crt}, the elements in the set $N_P$ does not have to be distinct if for each $P\in S$ and $f\in N_P$, there are  monic polynomials            $g(T)\in \mathcal{O}_{P}[T]$ which are pairwise distinct and irreducible over $k(P)$ of $\deg g(T)=f$.   
\end{rem} 
\begin{lem} \label{condition}  Let $E/F$ and $F'/F$ be finite separable extensions of function fields in some algebraic closure of $F$. Suppose that                           $\mathbb{F}_q$ is algebraically closed in $F$ and $F'$, and there is a place $P$ of $F$ that is totally ramified in $E/F$ and unramified in $F'/F$.                Then $E/F$ and $F'/F$ are linearly disjoint and $\mathbb{F}_q$ is algebraically closed in $EF'$. 
\end{lem}
\begin{proof} The linear disjointness follows from the existence of $P$ and Abhyankar's Lemma \cite{Stichtenoth}.  Let $L/\mathbb{F}_q$ be a finite extension of                $\mathbb{F}_q$. Then $P$ is unramified in the constant field extension $F'L$. Hence, again by applying Abhyankar's Lemma, we obtain that $EF'/F'$ and              $F'L/F'$ are linearly disjoint, and so  
             \[EF'\cap F'L=F'.\]
             This gives that $EF'\cap L=\mathbb{F}_q$, as $\mathbb{F}_q$ is algebraically closed in $F'$. Since this holds for any finite extension                             $L/\mathbb{F}_q$, we obtain that $\mathbb{F}_q$ is algebraically closed in $EF'$.
\end{proof}
\begin{lem} \label{lem-1} Let $F/\mathbb{F}_q$ be an algebraic function field and let $E$, $F'$ and $E'$ be finite separable extensions of $F$ such that $E'=EF'$. Suppose that $E/F$ and $F'/F$ are linearly disjoint.   
\begin{itemize}
              \item[(i)]  Set $E:=F(y)$, $m:=[E:F]$, and consider the set
              \[ M:=\big\{ P\in \Places(F): \{1,y,\ldots,y^{m-1}\} \textrm{ is an integral basis for $E/F$ at $P$}\big\}.\]
              Let $P\in M, P'\in \Places(F')$ with $P'|P$. Suppose that  $e(P'|P)$ is coprime to any ramification index of $P$ in $E$. Then  above $P'$ and each $Q \in \Places(E)$ with $Q|P$ there are exactly $gcd(f(Q|P),f(P'|P))$ places $Q'\in \Places(E')$, and moreover, for each such place $Q'$,
              \begin{equation}\label{relativedeg}
              k(Q')=k(Q)k(P').
              \end{equation} 
              \item[(ii)] Set $n:=[F':F]$. Then 
              \[g(E')\leq mg(F')+ng(E)-nmg(F)+(n-1)(m-1).\]
\end{itemize}
\end{lem}
\begin{proof} (i)  We first note that by \cite[Theorem  3.3.6]{Stichtenoth}, the set $M$  contains almost all places of $F$. Fix a place $P\in M$ with an extension $P'$ in $E'$ satisfying the given assumption. Let $\varphi(T)\in \OO_P[T]$ be the minimal polynomial of $y$ over $F$ and 
\begin{equation}
\Bar{\varphi}(T)=\prod_{i=1}^r \bar{g_i}(T)^{\epsilon_i}
\end{equation}
be the decomposition of $\bar{\varphi}(T)$ into irreducible factors over $k(P)$. Then by Kummer's Theorem \cite{Stichtenoth}, for $1\leq i\leq r$, there are places $Q_i\in \Places(E)$ satisfying 
\begin{equation}\label{kummer} 
Q_i|P,\; g_i(y)\in Q_i,\; e(Q_i|P)=\epsilon_i,\; f(Q_i|P)=\deg g_i,
\end{equation}
and these are all extensions of $P$ in $E$. For  each $1\leq i\leq r$, set
\begin{equation}\label{reldegeqn} 
k_i:=[k(Q_i)k(P'):k(P')].
\end{equation}
Then as $\bar{g}_i(T)$ is irreducible over $k(P)$, it is separable, and so 
\[ \bar{g}_i(T)^{\epsilon_i}=\prod_{j=1}^{s_i}\bar{h}_{ij}(T)^{\epsilon_i}\in  k(P')[T],\]
where $\bar{h}_{i1}(T),\ldots,\bar{h}_{is_i}(T)$ are pairwise distict, monic, irreducible polynomials in $k(P')[T]$ of $\deg \bar{h}_{ij}(T)=k_i$ for all $1\leq j\leq s_i$, and 
\begin{equation}\label{numberofplaces}
s_i=gcd(f(Q_i|P),f(P'|P)), 
\end{equation}
Again by Kummer's Theorem, for $1\leq j\leq s_i$, there are places $Q_{ij}\in \Places(E')$ satisfying 
\begin{equation}\label{eq-1}
Q_{ij}|P',\; h_{ij}(y)\in Q_{ij},\; f(Q_{ij}|P')\geq \deg h_{ij}=k_i.
\end{equation}
Moreover, as $h_{ij}(T)\mid g_i(T)$, it follows that each $Q_{ij}|Q_i$. We need to prove that these are all extensions of $Q_i$ and $P'$, then  the first part of assertion (i) follows, by (\ref{numberofplaces}).
Since by assumption $e(Q_i|P)$ and $e(P'|P)$ are coprime, it follows from Abhyankar's Lemma\cite{Stichtenoth} that 
\begin{equation} \label{eq-2}
e(Q_{ij}|P')=e(Q_i|P)=\epsilon_i \textrm{ for  all $1\leq j\leq s_i$.}
\end{equation}
 As this holds for all $1\leq i\leq r$, by using (\ref{kummer}), (\ref{eq-1}) and (\ref{eq-2}), we obtain that 
\begin{eqnarray*}
[E':F']=\sum_{\substack{Q\in \Places(E)\\ Q|P}}e(Q|P)f(Q|P)=\sum_{i=1}^{r}\epsilon_i\deg g_i(T)=\sum_{i=1}^r \epsilon_i\sum_{j=1}^{s_i} k_i\qquad\qquad\\
 \qquad\leq\;\sum_{i=1}^r\sum_{j=1}^{s_i} e(Q_{ij}|P')f(Q_{ij}|P')\leq\sum_{\substack{Q'\in \Places(E')\\Q'|P'}} e(Q'|P')f(Q'|P')=[E':F'].
\end{eqnarray*}
Hence, for each  $1\leq i\leq r$,  we get 
\begin{equation} \label{reldegeq}
 f(Q_{ij}|P')=k_i \textrm{ for $1\leq j\leq s_i$,}
\end{equation}
and $Q_{i1},\ldots Q_{is_i}$ are all places of $E'$ lying over $Q_i$ and $P'$. Then (\ref{relativedeg}) is clear, by (\ref{reldegeq}).\\
(ii)  We first claim that  $d(Q'|P')\leq e(Q'|Q)d(Q|P)$ for any $Q'\in \Places(E')$, $P':=Q'\cap F'$, $Q:=Q'\cap E$ and $P:=Q'\cap F$. Using this claim, the proof of (ii) will follow: 
 \begin{eqnarray*}
 Diff(E'/F')&=&\sum_{P'\in \Places(F')}\sum_{\substack{Q'|P'}} d(Q'|P')Q'\leq \sum_{\substack{Q\in \Places(E)\\ P=Q\cap F}}\sum_{\substack{ Q'|Q}} e(Q'|Q)d(Q|P)Q'\\
           &=&Con_{E'/E}(Diff(E/F)),
\end{eqnarray*}
where $Con_{E'/E}$ is the conorm map and $Diff$ is the the different. Hence, by \cite[Corollary 3.1.14]{Stichtenoth},
\begin{equation} \label{conorm}
\deg (Diff(E'/F'))\leq [E':E]\deg(Diff(E/F)).
\end{equation}
 By using (\ref{conorm}) and the Hurwitz Genus Formula for the extensions $E'/F'$ and $E/F$, we get 
 \begin{eqnarray*}
2g(E')-2&\leq& m(2g(F')-2)+n\deg(Diff(E/F))\\
        &=&m(2g(F')-2)+n\big(2g(E)-2-m(2g(F)-2)\big),
\end{eqnarray*}        
from which (ii) follows.

 Now to prove the claim,  consider the completions $\hat{F},\hat{E},\hat{F'}$ and $\hat{E'}$ with respect to the places $P,Q,P'$ and $Q'$. Since the different is preserved by completion, see \cite[p.52, Proposition 10]{Serre}, it sufficies to prove it in the comleted setting.  By \cite[p.57, Proposition 12]{Serre}, there is $\alpha \in \hat{E}$ such that $\mathcal{O}_{\hat{Q}}=\mathcal{O}_{\hat{P}}[\alpha]$.  Let $f(T)\in \mathcal{O}_{\hat{P}}[T]$, (resp. $g(T)\in \mathcal{O}_{\hat{P'}}[T]$)  be the minimal polynomial of $\alpha$ over $\hat{F}$ (resp. over $\hat{F'}$). By Gauss Lemma \cite{Hungerford}, we can write $f(T)=g(T)h(T)$ in $\mathcal{O}_{\hat{P'}}[T]$, then 
 \[ f'(\alpha)=g'(\alpha)h(\alpha). \]
Thus, by using \cite[p.56, Corollary 2]{Serre}, the desired result follows:
 \[d(\hat{Q'}|\hat{P'})\leq v_{\hat{Q'}}(g'(\alpha))\leq v_{\hat{Q'}}(f'(\alpha))=e(\hat{Q'}|\hat{Q})v_{\hat{Q}}(f'(\alpha))=e(\hat{Q'}|\hat{Q})d(\hat{Q}|\hat{P}).\]
\end{proof}
\begin{thm} \label{thmlocal} With the same notations as in Lemma \ref{lem-1}, suppose that $\EE:=E\cdot \FF$ is a composite tower of $\FF/\mathbb{F}_q$. Let  $P\in M$ such that $e(Q|P)$ is coprime to any ramification index of $P$ in $\FF$, for all $Q\in \Places(E)$ with $Q|P$. Then for any $Q|P$ and $s\geq 1$, 
\[\nu_s(Q,\EE)=\frac{f(Q|P)}{s}\sum_{\substack{d\in \N \\ lcm(\deg Q,d)=s}} d\cdot \nu_d(P,\FF).\]
\end{thm}
\begin{proof} Set $E':=E_n$, $F':=F_n$, for any $n\geq 1$. By Lemma \ref{lem-1}(i), there are $gcd(f(Q|P),f(P'|P))$ places $Q'\in \Places(E')$ above any fixed $Q,P'$ lying over $P\in M$, and moreover for each such place $Q'$,
\[ k(Q')=k(Q)k(P').\]
In particular, $s:=f(Q'|P)=lcm((f(Q|P),f(P'|P))$, and so $d:=f(P'|P)$ divides $s$. Conversely, for any $P'|P$ with $d=f(P'|P)$ such that $s=lcm(f(Q|P),d)$ has an extension $Q'$ in $E'$ with $f(Q'|P)=s$. Hence,
\begin{eqnarray*}
\sum_{\substack{Q'|Q\\f(Q'|P)=s}} 1&=&\sum_{\substack{d\in \N\\ lcm(f(Q|P),d)=s}}\sum_{\substack{P'|P\\\; f(P'|P)=d}} \;\sum_{\substack{Q'|P'\\Q'|Q}}1\\
                                 &=&\sum_{\substack{d\in \N\\ lcm(f(Q|P),d))=s}} \sum_{\substack{P'|P\\\;f(P'|P)=d}} gcd(f(Q|P),d)\\
                                 &=& \sum_{\substack{ d\in \N\\ lcm(f(Q|P),d)=s}} \frac{d\cdot f(Q|P)}{s}\sum_{\substack{P'|P\\ f(P'|P)=d}} 1.
\end{eqnarray*} 
Since in general $lcm(af,ad)=as$ if and only if $lcm(f,d)=s$, we can write the summation indices in terms of absolute degrees instead of relative degrees with respect to $P$ as base place, and obtain 
\begin{eqnarray*}
B_s(Q,E')&=&\sum_{\substack{Q'|Q\\ \deg Q'=s}} 1= \frac{f(Q|P)}{s}\sum_{\substack{d\in \N\\ lcm(\deg Q,d)=s}}d\sum_{\substack{P'|P\\\;\deg P'=d}} 1\\
         &=&\frac{f(Q|P)}{s} \sum_{\substack{d\in \N\\ lcm(\deg Q,d)=s}} d\cdot B_d(P,F').
\end{eqnarray*}
Dividing by $[E':E]$ and then taking the limit as $n\rightarrow \infty$ proves the theorem.
\end{proof}
\begin{proof}[Proof of Theorem \ref{mainthm}] It is enough to prove (i), then (ii) is immediate.
 By applying Lemma \ref{condition} and Proposition \ref{crt} with the set $S:=Supp(\FF)$ and $N_P=N$ for each $P\in Supp(\FF)$, one can construct an extension $E/F$ such that $\EE=E\cdot \FF$ is a composite tower of $\FF$ and for each $f\in N$, any $P\in S$ has exactly one extension $Q$ in $E$ with $f(Q|P)=f$ and these are the only extensions of $P$ in $E$. Moreover, by the construction of $E/F$, all places $P\in S$ are unramified in $E$ and $S$ is contained in the set $M$ defined in Lemma \ref{lem-1}.
 
By Corollary \ref{corcomp} and the construction of $E/F$, and Theorem \ref{thmlocal}, the statement (\ref{stm-1}) is immediate. Therefore,  for any $s\geq 1$, by using Theorems \ref{thmlocal} and \ref{thmprem}(v), we get 
\begin{eqnarray*}
\nu_s(\EE)&=&\sum_{\substack{f\in N, d\in \N \\ P\in S}} \sum_{\substack{ f(Q|P)=f\\ lcm(\deg Q,d)=s}} \nu_s(Q,\EE)\\
          &=& \sum_{\substack{f\in N\\ d\in \N}} \frac{f}{s}\sum_{\substack {P\in S \\ lcm(f\deg P,d)=s}} d\cdot \nu_d(P,\FF).
\end{eqnarray*} 
\end{proof}

          We note here that in the case that $Supp(\mathcal{F})$ is infinite, one can apply Theorem \ref{mainthm} with a finite subset $S\subseteq                          Supp(\mathcal{F})$ and gets a finite subset of $Supp(\mathcal{E})$. 

        As there are many towers over a given finite field $\mathbb{F}_q$ with non empty finite support, such as many of the class field towers and the recursive            towers (see Section 4), Theorem \ref{mainthm} can be often applied. More specifically, since there are many towers $\mathcal{F}$  over $\mathbb{F}_{q^2}$ attaining  the Drinfeld-Vladut bound of order one, i.e., $\mathcal{P}(\mathcal{F})=\left\{1\right\}$ (see Example \ref{exm}), by using Theorem \ref{mainthm}(ii) one gets immediately the following consequence:
\begin{cor}\label{maincor} For any given set $M\subset \mathbb{N}$, there exists a tower of function fields $\mathcal{E}$ over $\mathbb{F}_{q^2}$ with 
            \[ \mathcal{P}(\mathcal{E})\cap M=\emptyset.\]
\end{cor} 
\begin{proof}  Let $N\subseteq \mathbb{N}\setminus M$ be a finite set. Consider a tower $\mathcal{F}$ over $\mathbb{F}_{q^2}$ with $\mathcal{P}(\mathcal{F})=\left\{1 \right\}$. Then  by  Theorem \ref{mainthm}(ii), there is a composite tower $\mathcal{E}$ of $\mathcal{F}$ over $\mathbb{F}_{q^2}$  with 
              \[\mathcal{P}(\mathcal{E})=N,\]
 and hence the corollary follows.
\end{proof} 

\subsection{\normalsize{Computation of the genus of a composite tower and an aplication}}
\paragraph{}

In this part, we assume that $\mathcal{E}:=E\cdot \mathcal{F}$ is a composite tower over $\mathbb{F}_q$ with $E/F$ . We begin with the computation of the genus $\gamma(\mathcal{E})$ of the tower $\mathcal{E}$, under certain conditions. We first note that for a tower $\mathcal{F}=(F_n)_{n\geq 0}$ over $\mathbb{F}_q$ if the set 
\[R:=\left\{ P\in \mathbb{P}(F): \textrm{ $P$ is ramified in $F_n$ for some $n\geq 1$}\right\}\]
 is finite, then by \cite[Lemma 3.4]{GST097}, the following limit exists:
 \[ \alpha(\mathcal{F}):=\lim_{n\rightarrow \infty} \frac{\deg A_n}{[F_n:F]}\quad \textrm{ where } A_n:=\sum_{\substack{P\in \mathbb{P}(F_n)\\
 P\cap F\in R}} P.\]
\begin{thm}\label{genus} Set $m:=[E:F]$. For the genus $\gamma(\mathcal{E})$ the following hold:
\begin{itemize}
\item[(i)] $m\gamma(\mathcal{F})\leq \gamma(\mathcal{E})\leq g(E)-1+m(1-g(F)+\gamma(\mathcal{F}))$. 

If furthermore all $P\in R$ are unramified in $E$, then the second equality holds. 
\item[(ii)] If $R$ is finite, $\alpha(\mathcal{F})=0$ and all $P\in R$ are tame in $E$, then
\[ \gamma(\mathcal{E})=g(E)-1-\delta/2+m(1-g(F)+\gamma(\mathcal{F})),\]
where $\delta:=\sum\limits_{\substack{Q \in \mathbb{P}(E)\\ Q\cap F\in R}}d(Q|Q\cap F)\cdot \deg Q.$
 \end{itemize}
 \end{thm}
 \begin{proof} The assertion (ii) is given in \cite[Theorem 3.6]{GST097}. The proof of the second part of (i) is with minor modifications the same as that of (ii).
 Hence, we need just to prove the first part of (i).  By using the Hurwitz Genus Formula \cite{Stichtenoth}, one can easily conclude that
\begin{equation*}
                   g(E_n)\geq mg(F_n)-m \textrm{ for all $n\geq 0$, and so $\gamma(\mathcal{E})\geq  m\gamma(\mathcal{F})$.}
\end{equation*} 
Next, to prove the second inequality  we apply Lemma \ref{lem-1}(ii), with $E':=E_n$, $F':=F_n$, for any $n\geq 1$, which yields
\[g(E')\leq mg(F')+[F':F]g(E)-m[F':F]g(F)+([F':F]-1)(m-1).\]
Dividing by $[E':E]=[F':F]$ and then taking the limit as $n\rightarrow \infty$ gives the assertion. 
\end{proof}
\begin{rem} In the second part of Theorem \ref{genus}(i), when $g(F)=0$, Castelnouvo's Inequality \cite{Stichtenoth} holds for the function fields $E/\mathbb{F}_q, F_n/\mathbb{F}_q$ with their compositum $E_n$. 
\end{rem}
\begin{exm}\label{exm} Let $N\subset \mathbb{N}$ be a finite set and set $m:=\sum_{f\in N}f$. Consider the tower $\mathcal{F}=(F_n)_{n\geq 0}$ over $\mathbb{F}_{q^2}$, with $(m,q)=1$, which is studied in \cite{GS095}.  The tower $\mathcal{F}$ is defined by $F=\mathbb{F}_{q^2}(x_0)$, and $F_{n+1}=F_n(x_{n+1})$, where $x_{n+1}$ satisfies the equation
\[x_{n+1}^qx_{n}^{q-1}+x_{n+1}=x_n^q.\]
This tower has the following properties:
\begin{itemize}
\item $Supp(\mathcal{F})=\{ P\in \mathbb{P}(F)|\textrm{ $x_0(P)=\alpha$ for some $0\neq \alpha \in \mathbb{F}_{q^2}$} \}$,
\[ \mathcal{P}(\mathcal{F})= \left\{ 1\right\} \textrm{ and } \nu_1(P,\mathcal{F})=1 \textrm{ for all $P \in Supp(\mathcal{F})$.}\qquad\qquad \] 
\item $R:=\left\{P_0,P_\infty \right\}\subseteq\mathbb{P}(F)$, where $P_0$ (resp. $P_\infty$) is the zero (resp. the pole) of $x_0$, is the set of ramified places in $\mathcal{F}$. 
\item $P_\infty$ is totally ramified in $\mathcal{F}$, and $\gamma(\mathcal{F})=q+1$.
\item $\beta_1(\mathcal{F})=q-1$, i.e., $\mathcal{F}$ attains the Drinfeld-Vladut bound of order one.
\item $\alpha(\mathcal{F})=0$, see \cite[Lemma 2.9]{GS095} or \cite[ Example 3.8(v)]{GST097}.
\end{itemize}
Let $E:=F(y)$ with $y$ a root of the polynomial 
 \[\varphi(T):=\prod_{f\in N}g_f(T)-x_0^{q^2}+x_0\in F[T],\]
where $g_f\in \mathbb{F}_{q^2}[T]$ is a monic, irreducible polynomial of $\deg g_f(T)=f$. Then the following hold: 
\begin{itemize}
\item[(i)] $P_\infty$ is clearly totally ramified and by Kummer's Theorem $P_0$ is unramified in $E$. Thus, $E/F$ is a separable extension of degree $[E:F]=m$.
\item[(ii)] Since $[F_{n+1}:F_n]=q$ and $(m,q)=1$, by using  Abhyankar's Lemma \cite{Stichtenoth}, we obtain that $P_\infty$ is totally ramified in $E_n:=EF_n$ for all $n\geq 1$. Thus, $E/F$ and $F_n/F$ are linearly disjoint and $\mathbb{F}_{q^2}$ is algebraically closed in $E_n$. Hence, $\mathcal{E}:=E\cdot \mathcal{F}$ is a tower over $\mathbb{F}_{q^2}$. 
\item[(iii)] $E/\mathbb{F}_{q^2}(y)$ is an elementary abelian extension. Then one can easily conclude that only the pole of $y$, say $Q_\infty$, is ramified in $E/\mathbb{F}_{q^2}(y)$, with the extension $Q'_\infty$. Moreover, $e(Q'_\infty|Q_\infty)=q^2$ and $d(Q'_\infty|Q_\infty)=(m+1)(q^2-1)$. Now it follows from the Hurwitz Genus Formula\cite{Stichtenoth} for the extension $E/\mathbb{F}_{q^2}(y)$ that the genus of $E$ is
\[g(E)=\frac{(m-1)(q^2-1)}{2}.\]
\item[(iv)] By Kummer's Theorem, each place $P\in Supp(\mathcal{F})$ is unramified in $E$ and has exactly one extension of degree $f$ in $E$, for each $f\in N$. \item [(v)] All places $P\in Supp(\mathcal{F})$ split completely in $\mathcal{F}$, as $\nu_1(P,\mathcal{F})=1$. 
\end{itemize}
By combining (i), (ii), (iii) and applying Theorem \ref{genus}(ii),  we obtain that
\begin{equation}\label{newgenus}
\gamma(\mathcal{E})=\frac{m(q^2+2q+2)-q^2}{2}.
\end{equation}
Now since the extension $E/F$ has all properties in Proposition \ref{crt}, by using Theorems \ref{thmlocal} and \ref{mainthm}(ii),  we obtain that 
\[ Supp(\mathcal{E})=\left\{Q\in \mathbb{P}(E):\; Q\cap F\in Supp(\mathcal{F}) \right\}, \; \mathcal{P}(\mathcal{E})=N,\textrm{ and }\]
\[ \nu_f(Q,\mathcal{F})=1\textrm{ for all $Q\in Supp(\mathcal{E})$ with some $f\in N$, and so  $\nu_f(\mathcal{E})=q^2-1$}.\]
Then by Theorem \ref{thmprem}(iv) and (\ref{newgenus}), we get
\[ \beta_f(\mathcal{E})=\frac{2(q^2-1)}{m(q^2+2q+2)-q^2} \textrm{ for all $f\in N$. } \]
\end{exm}
\begin{rem} In Example \ref{exm}, the \textit{deficiency}
\[\delta(\mathcal{E})=1- \frac{2(q^2-1)}{m(q^2+2q+2)-q^2}\sum_{f \in N} \frac{f}{q^{f}-1},\]
which depends on $m,q$ and the set $N$. Thus, by  an appropriate choice of $m,q$ and $N$, one can construct many different towers $\mathcal{E}/\mathbb{F}_{q^2}$ with distinct $\delta$.
\end{rem}

\section{Examples}

 \subsection{\normalsize{Non-optimal recursive towers with all but one invariants zero}}
\paragraph{}
 We first recall that a \textit{non-optimal} tower means; a tower which does not attain the generalized Drinfeld-Vladut bound given in Theorem \ref{thmprem}(ii). We begin with some simple remarks, which we will apply in the subsequent examples.
 \begin{rem}\label{Hasegawa.rem}
Let $\mathcal{F}=(F_n)_{n\geq 0}$. For any $n\geq 1$, we have that
\[ r B_r\left(F_n/\mathbb{F}_q\right)=\sum _{d|r} \mu(\frac{r}{d})B_1(F_n\mathbb{F}_{q^d}/\mathbb{F}_{q^d}),\]
where $\mu$ denotes the Mobius function (see \cite[p.207] {Stichtenoth}). Therefore,
\[ r \beta_r\left(\mathcal{F}/\mathbb{F}_q\right)=\sum _{d|r} \mu(\frac{r}{d})\beta_1(\mathcal{F}\mathbb{F}_{q^d}/\mathbb{F}_{q^d}).\]
\end{rem}
\begin{rem}\label{mu} For any $t\geq 1$ we have 
\begin{eqnarray*}
\sum_{t|d|r} \mu(\frac{r}{d})=\begin{cases} 
                                     1 & \textrm{ if $r=t$,}\\
                                     0 & \textrm{ else.} 
                           \end{cases}
\end{eqnarray*}
\end{rem}                          
\begin{proof} We know that 

\begin{eqnarray*}
\sum_{d|r} \mu(d)=\begin{cases} 
                                     1 & \textrm{ if $r=1$,}\\
                                     0 & \textrm{ else.} 
                           \end{cases}
\end{eqnarray*}
Clearly, if $t\nmid r$, then there is nothing to prove. So we assume that $r=t^ns$ for some $n\geq 1$ and $t\nmid s$, where $s$ is an integer.
We set $d:=tk$ where $k$ is a factor of $\frac {r}{t}$.  Then 
\begin{eqnarray*}
\sum_{t|d|r} \mu(\frac{r}{d})&=&\sum_{tk|t^ns}\mu \left(\frac{r}{d}\right)=\sum_{k|t^{n-1}s}\mu\left(\frac{t^{n-1}s}{k}\right)\\
                             &=& \sum_{k|t^{n-1}s} \mu\left(k\right)=\begin{cases} 
                             1 & \textrm{ if $t^{n-1}s=1$}\\
                             0 & \textrm{ else.} 
                  \end{cases}
\end{eqnarray*}
Hence, since $r=t^ns$, the result follows. 
\end{proof}
\begin{exm} \label{non-optimal1} Let $\mathcal{F}$  be the tower defined by the equation $y^2+y=x+1+1/x$ 
over a finite field $\mathbb{F}_{2^e}$ for some $e\geq 1$. Then by Example 5.8 in \cite{B04}, we have 
\begin{eqnarray*}
        \beta_1(\mathcal{F})= \begin{cases}
                                 3/2  & \textrm{ if $3$ divides $e$, } \\
                                  0   &  \textrm{ else. }
                      \end{cases}
\end{eqnarray*}  
Now we consider the tower $\mathcal{F}$ over $\mathbb{F}_q$, with $q=2^e$ where $3 \nmid e$.  Then by applying  Remark \ref{Hasegawa.rem}, we obtain that
\[ r \beta_r(\mathcal{F}/\mathbb{F}_q)= \sum_{3|d|r} \mu(\frac{r}{d})\beta_1(\mathcal{F}\mathbb{F}_{q^d}/\mathbb{F}_{q^d})=\frac{3}{2}\sum_{3|d|r} \mu(\frac{r}{d}).\]
Hence, if $r$ is not divisible by $3$, then clearly $\beta_r(\mathcal{F}/\mathbb{F}_q)=0$. Now suppose that $3$ divides $r$. By applying Remark \ref{mu} with $t=3$, the following holds:
\[\PP(\FF/\GF_q)=\big\{3\big\} \textrm{ with } \beta_3(\mathcal{F}/\mathbb{F}_q)=\frac{1}{2}.\]
\end{exm}
\hspace{-0.6cm}This example implies that for $q=2^e$ where $3\nmid e$, we have 
\[A_3(q)\geq \frac{1}{2}.\]
Notice that for $q=2$, we get a lower bound close to the Drinfeld-Vladut bound of order $3$ with the \textit{deficiency}  $\delta(\mathcal{F}/\mathbb{F}_2)=0.17962$. 
\begin{exm} \label{non-optimal2}Let $q=3^e$ for some $e\geq 1$ and $\mathcal{F}$ be the tower given in \cite{GSR003}, which is defined by the equation 
$y^2=\frac{x(x-1)}{x+1}\textrm{ over $\mathbb{F}_q$}$.
Then by Example 2.4.3 in \cite{H2007}, we have
 \begin{eqnarray*}
        \beta_1(\mathcal{F})= \begin{cases}
                                 2/3  &  \textrm{ if $e$ is even,}\\
                                 0  & \textrm{ if $e=1$.}                      
\end{cases}
\end{eqnarray*} 
 Now by applying Remark \ref{mu} with $t=1$, we obtain that
 \[ \PP(\mathcal{F}/\mathbb{F}_{9})=\big\{1\big\} \textrm{ with } \beta_1(\mathcal{F}/\mathbb{F}_{9})=\frac{2}{3}.\]
 Thus, its \textit{deficiency}  $\delta(\mathcal{F}/\mathbb{F}_9)=\frac{2}{3}\approx 0.66$.
\end{exm}
\begin{exm} \label{non-optimal3} Let $q$ be a power of the prime number $3$, and $\mathcal{F}$ be the tower given in \cite{GSR003}, which is defined by the equation $y^2=\frac{x(x+1)}{x-1} \textrm{ over $\mathbb{F}_q$}.$
Then by a remark in \cite[p.46]{H2007}, we have 
\[\beta_r(\mathcal{F}/\mathbb{F}_{81^n})=2\textrm{ for all $n\geq 1$}.\]
Now by applying Remark \ref{mu} with $t=1$, we get that
 \[ \PP(\mathcal{F}/\mathbb{F}_{81})=\big\{1\big\} \textrm{ with } \beta_1(\mathcal{F}/\mathbb{F}_{81})=2,\]
 and so  the \textit{deficiency} $\delta(\mathcal{F}/\mathbb{F}_{81})=0.75$.
\end{exm}
\begin{exm} \label{non-optimal4} Let $p\geq 3$ be a prime number and $\mathcal{F}$ be the tower over $\mathbb{F}_{p^e}$ defined by the equation $y^2=(x^2+1)/ {2x}$. Then by Example 5.9 in \cite{B04}, we have 
\begin{eqnarray*}
        \beta_1(\mathcal{F})= \begin{cases}
                                 p-1  & \textrm{ if $2$ divides $e$, } \\
                                  0   &  \textrm{ else. }
                      \end{cases}
\end{eqnarray*}  
 We consider the tower $\mathcal{F}$ over $\mathbb{F}_{q}$, where $q:=p^e$ with $2$ not dividing $e$. Then, by applying Remark \ref{mu} with $t=2$, we obtain that
\[ \PP(\mathcal{F}/\mathbb{F}_{q})=\big\{2\big\} \textrm{ with } \beta_2(\mathcal{F})=\frac{p-1}{2}.\]
Thus, the tower $\mathcal{F}$ over $\mathbb{F}_p$ attains the Drinfeld-Vladut bound of order $2$. 
\end{exm} 

\begin{cor} In the following cases there exists a non-optimal recursive tower over $\mathbb{F}_q$ with all but one invariants zero:
\begin{itemize}
\item[(i)] $q= 2^e$ with $3$ not dividing $e$,
\item[(ii)] $q=3^e$ with $e=2$ or $4$,
\item[(iii)] $q=p^e$ with $p\geq 3$, $e>2$ and $2$ not dividing $e$. 
\end{itemize}
\end{cor}
\begin{proof} See Examples \ref{non-optimal1}, \ref{non-optimal2}, \ref{non-optimal3},  and \ref{non-optimal4}, respectively. 
\end{proof}

\subsection{Recursive towers attaining the Drinfeld-Vladut bound of order $r$}
\begin{exm} \label{GS-2} Let $q^r$ be a square and $\mathcal{F}$ be the tower defined by the equation 
\begin{equation}\label{2ndtower}
y^{q^{r/2}}+y=\frac{x^{q^{r/2}}}{x^{q^{r/2}-1}+1}
\end{equation}
over a finite field $\mathbb{F}_{p^e}$, for some $e\geq 1$. 
Then by Example 5.7 in \cite{B04}, we have 
\begin{equation*}
        \beta_1(\mathcal{F})= \begin{cases}
                                 q^{r/2}-1  & \textrm{ if  } \mathbb{F}_{q^r}\subseteq \mathbb{F}_{p^e}, \\
                                  0   &  \textrm{ else }
                      \end{cases}
\end{equation*}                           
Now consider the tower $\mathcal{E}/\mathbb{F}_{q^r}$ defined by (\ref{2ndtower}), which is studied in \cite{GS096}. Then from Remark \ref{Hasegawa.rem}, we get 
 \[r \beta_r(\mathcal{E}/\mathbb{F}_q)=\sum_{d|r}\mu(\frac{r}{d})\beta_1(\mathcal{E}/\mathbb{F}_{q^d}),\]
 which implies that
     \[\beta_r(\mathcal{E}/\mathbb{F}_q)=\frac{q^{r/2}-1}{r}=A_r(q). \]
\end{exm}
\begin{exm} \label{GS-1} The tower $\mathcal{T}$ defined by the equation 
\[y^{q^{r/2}} x^{q^{r/2}-1}+y=x^{q^{r/2}}\]
 over $\mathcal{F}_{q^r}$, with $q^r$ a square, is optimal and from \cite[Remark 3.11, Corollary 2.4]{GS096}, we have that $\beta_1(\mathcal{E})\geq\beta_1(\mathcal{T})$, and so
\begin{eqnarray*}
        \beta_1(\mathcal{T})= \begin{cases}
                                 q^{r/2}-1  & \textrm{ over } \mathbb{F}_{q^r}, \\
                                  0         &  \textrm{ over $\mathbb{F}_{p^e}$ where $\mathbb{F}_{q^r}\nsubseteq \mathbb{F}_{p^e}$.}
                             \end{cases}
\end{eqnarray*} 
Then with the same reason as in the previous example we get that
\[\beta_r(\mathcal{T}/\mathbb{F}_q)=\frac{q^{r/2}-1}{r}=A_r(q). \]
\end{exm}
Now from Examples \ref{GS-2} and \ref{GS-1} the following is immediate:
\begin{cor} \label{orderr} For any $r\geq 1$ and any prime power $q$ such that $q^r$ is a square, there exists a tower of function fields over $\mathbb{F}_{q^r}$ attaining the generalized Drinfeld-Vladut bound of order $r$. 
\end{cor}

\section{ \normalsize{Some remarks on $A_r(q)$}}
\paragraph{}
 Here we consider an exact sequence $\mathcal{F}=(F_n)_{n\geq 0}$ of function fields over $\mathbb{F}_q$ with its constant field extension $\mathcal{F}\mathbb{F}{q^r}$ of degree $r$ for some $r\geq 1$. 
\begin{lem}\label{exact}
 \[\beta_1(\mathcal{F}\mathbb{F}_{q^r})\geq r  \beta_r(\mathcal{F}).\]
\end{lem}
\begin{proof}
Since $F_n\mathbb{F}_{q^r}$ is a constant field extension of $F_n/\mathbb{F}_{q}$, we have that 
\begin{equation}\label{usufuleqn1}
B_1(F_n\mathbb{F}_{q^r})=\sum_{i|r} i B_i(F_n)\textrm{ and } g(F_n\mathbb{F}_{q^r})=g(F_n),
\end{equation}
and therefore
\begin{eqnarray*}
 \beta_1(\mathcal{F}\mathbb{F}_{q^r})&=&\lim_{n\rightarrow \infty}\frac{B_1(F_n\mathbb{F}_{q^r})}{g(F_n\mathbb{F}_{q^r})}=\lim_{n\rightarrow \infty}\frac{1}{g(F_n)} \left(\sum_{i|r} i B_i(F_n)\right)\\
     &\geq & \lim_{n\rightarrow \infty} \frac{r  B_r(F_n)}{g(F_n)}=r \beta_r(\mathcal{F}).
\end{eqnarray*}
\end{proof}
By using Lemma \ref{exact},  for any integer $r\geq 1$ and prime power $q$, one gets
\begin{cor} $A(q^r)\geq r A_r(q)$.
\end{cor}
This fact might be well-known, but we could not find any reference in the literature.

Next, one can easily conclude from Lemma \ref{exact} that if for some $r\geq 1$ the sequence $\mathcal{F}/\mathbb{F}_q$ attains the Drinfeld-Vladut bound of order $r$, then the sequence $\mathcal{F}\mathbb{F}_{q^r}/\mathbb{F}_{q^r}$ attains the classical Drinfeld-Vladut bound, i.e., of order one. Furthermore, in that case we have 
\[q^{r/2}-1\geq A(q^r)\geq \beta_1(\mathcal{F}\mathbb{F}_{q^r})\geq r \beta_r(\mathcal{F})=r A_r(q)=r \left( \frac{q^{r/2}-1}{r} \right)=q^{r/2}-1,\]
which implies that 
\begin{equation}\label{rA_r(q)=A(q^r)}
A(q^r)=r A_r(q)=q^{r/2}-1.
\end{equation}
We note here that it follows from  Corollary \ref{orderr} that for any square $q^r$, the equality $(\ref{rA_r(q)=A(q^r)})$ holds. However, in the case that $q^r$ is not a square, it is not known whether there exists any  asymptotically exact sequences of algebraic function fields over a finite field $\mathbb{F}_q$ attaining  the Drinfeld-Vladut bound of order $r$. Therefore, the following question arises:\\\\
\textbf{Problem:} Are there any integers $r\geq 2$ and any prime power $q$, except when $q^r$ is a square, such that the following holds:
\[ A(q^r)=r A_r(q) ?\]

\end{document}